\documentclass[12pt,a4paper]{elsarticle}
\usepackage[utf8]{inputenc}
\usepackage[english]{babel}
\usepackage{amsmath,amssymb,amsthm,mathrsfs}
\usepackage{lmodern}
\DeclareMathOperator{\spann}{span}
\usepackage{hyperref}
\usepackage[hoffset=-.7cm,voffset=1cm,top=.5in,text={17cm,23cm}]{geometry}
\makeatletter
\renewcommand\subsection{\@startsection {subsection}{2}{\z@}%
                {3.5ex \@plus -1ex \@minus -.2ex}%
                {1.5ex \@plus.2ex\noindent}%
                {\normalsize\bfseries\itshape}%
                }
\makeatother
\def\z{\mathfrak{z}}
\def\g{\mathfrak{g}}
\def\d{\mathfrak{d}}
\def\s{\mathfrak{s}}

\def\R{\mathbb{R}}

\def\ad{\operatorname{ad}}

\def\alt{\raise1pt\hbox{$\bigwedge$}}

\theoremstyle{plain}
\newtheorem{teo}{\bf Theorem}[section]

\newtheorem{prop}[teo]{\bf Proposition}

\theoremstyle{definition}

\newtheorem{ejemplo}[teo]{\bf Examples}

\theoremstyle{remark}
\newtheorem{rem}[teo]{\bf Remark}

\begin{document}

\begin{frontmatter}

\title{Triangular structures on flat Lie algebras}

\author{Amine Bahayou}
\ead[url]{amine.bahayou@gmail.com}

%\author[Amine Bahayou]{Kasdi Merbah University}
%\ead[url]{amine.bahayou@gmail.com}

\begin{abstract}
In this work we study a large class of exact Lie bialgebras arising from noncommutative deformations of Poisson-Lie groups endowed with a left invariant Riemannian metric. We call these structures \emph{exact metaflat Lie bialgebras}. We give a complete classification of these structures. We show that given the metaflatness geometrical condition, these exact bialgebra structures arise necessarily from a nontrivial solution of the classical Yang-Baxter equation. Moreover, the dual Lie bialgebra is also flat and metaflat constituting an important kind of symmetry.
\end{abstract}


\end{frontmatter}

%\linenumbers
\section{Introduction}
Hawkins \cite{Ha1,Ha2} studied and introduced necessary conditions for the existence of a deformation of the graded algebra of differential forms $\Omega^\ast(M)$ on a Riemannian manifold $(M,g)$. He showed that the deduced Poisson bracket on the algebra of functions $C^\infty(M)$ extends to all forms and verifies the graded Jacobi identity if (\cite{Ha2}, Theorem 2.3, page 393):
\begin{enumerate}
\item The associated Levi-Civita contravariant connection $D$ is \emph{flat}.
\item The introduced \emph{metacurvature} tensor vanishes, ($D$ is \emph{metaflat}).
\end{enumerate}
We previously studied \cite{BB}, the so-called Hawkins conditions on a Poisson-Lie group endowed with a left invariant Riemannian metric, proved it is equivalent to (\cite{BB}, Theorem 3.1, page 446):
\begin{enumerate}
\item The dual Lie algebra $\g^\ast$ is flat, i.e. decomposes into an orthogonal sum $\g^\ast=S\oplus[\g^\ast,\g^\ast]$ (for the associated scalar product on $\g^\ast)$, where $S$ is an abelian subalgebra and the commutator ideal $[\g^\ast,\g^\ast]$ is even-dimensional and abelian.
\item The flat Lie bialgebra $(\g^\ast,\xi)$ is metaflat, i.e. for all $x,y\in S$ $\ad_x^2\xi(y)=0$, where $\ad_x$ is the extension of the adjoint representation to bivectors and $\xi$ is the $1$-cocycle defining the dual bialgebra structure of $\g^\ast$. (See the preliminary section for the definitions regarding Lie bialgbras). 
\end{enumerate}
The purpose of this paper is to show that any \textbf{exact} bialgebra structure on a flat (non abelian) Lie algebra arises necessarily from a nontrivial solution of the \textbf{classical} Yang-Baxter equation, provided the bialgebra is \textbf{metaflat}. We will show that the dual Lie bialgebra is also flat and metaflat when the starting Lie bialgebra is exact flat and metaflat.\par 
Let us describe the contents of the paper in more detail. In Section 2 we introduce the necessary notation and prove some preliminary results in Poisson geometry. In section 3, we present the main result of this paper, Theorem \ref{main}, and we discuss some examples in section 4.
\section{Preliminaries}\label{prelim}
For this section, we address the reader to the book \cite{V} where they can find all the material about Poisson geometry and Poisson-Lie groups. A Poisson structure on a smooth manifold $M$ is a Lie bracket $\{\cdot, \cdot\}$ on the space $C^{\infty}(M)$ of smooth functions on $M$ which satisfies the Leibniz rule. This bracket is called Poisson bracket and a manifold $M$ equipped with such a bracket is called Poisson manifold.
Therefore, a bivector field $\pi$ on $M$ such that the bracket
\[\{f,g\}:=\pi\left(df,dg\right)\]
is a Poisson bracket is called Poisson tensor or Poisson bivector field. A Poisson tensor can be regarded as a bundle map $\pi_{\sharp}: T^*M\to TM$:
\[\beta\left(\pi_{\sharp}(\alpha)\right)=\pi\left(\alpha,\beta\right).\]
A map $\phi: (M,\pi_M)\rightarrow (N,\pi_N)$ between two Poisson manifolds is called a Poisson map if for all $f,g\in C^{\infty}(N)$, one has
\begin{equation*}
\lbrace f\circ\phi, g\circ\phi\rbrace_M=\lbrace f,g\rbrace_N\circ \phi
\end{equation*}
\subsection{Lie bialgebras and Poisson Lie groups}
\label{sec:lie bialgebras}
A \emph{Poisson-Lie group} $(G,\pi)$ is a Lie group $G$ endowed with a Poisson tensor $\pi$ for which the multiplication $m : G \times G \rightarrow G$ is a Poisson map with respect to $\pi$ on $G$ and the product Poisson structure $\pi_{G\times G}=\pi\oplus\pi$ on $G\times G$. Equivalently, $\pi$ is multiplicative, i.e. for any $g,h\in G$\[\pi(gh)=(L_{g})_*\pi(h)+(R_{h})_*\pi(g),\]
where $(L_{g})_*$ (resp. $(R_{h})_*$) denotes the tangent map of the left translation of $G$ by $g$ (resp. the right translation of $G$ by $h$). \cite{LW}, page 503.\par 
The associated map $\pi_g:G\to\wedge^2\g$, defined by $\pi_r(g)=(R_{g^{-1}})_*\pi(g)$ satisfies
\[\pi_r(gh)=\pi_r(g)+\mathrm{Ad}_g\pi_r(h),\] 
which means that $\pi_r$ is a $1$-cocycle of $G$ with values in $\wedge^2\g$ for the Adjoint representation 
\[g\cdot\left(X\wedge Y\right):=\mathrm{Ad}_gX\wedge\mathrm{Ad}_gY.\]
The linear map $\xi : \g\to\wedge^2\g$ defined by $\xi(X)=\mathscr{L}_X\pi_r(e)$ is a $1$-cocycle with respect to the adjoint representation of $\g$ on $\g\wedge\g$
\[\ad_x(y\wedge z):=\ad_x y\wedge z+y\wedge\ad_x z.\]
which leads to the following definition:
A finite-dimensional Lie algebra $\g$ has a \emph{bialgebra structure} if it is endowed with a $1$-cocycle $\xi:\g\to\wedge^2\g$, i.e. for all $x,y$ in $\g$
\begin{equation}
\xi([x,y])=\ad_x\xi(y)-\ad_y\xi(x),
\label{1cocycle}
\end{equation}
and the transpose map $\xi^t=[\cdot,\cdot]_*:\g^\ast\wedge\g^\ast \to\g^\ast$ defined by
\[\langle\xi^t(\alpha\wedge\beta),x\rangle_{\g^\ast\times\g}:=\langle\alpha\wedge\beta,\xi(x)\rangle_{\wedge^2\g^\ast\times\wedge^2\g}\]
is a Lie bracket on $\g^\ast$.\\
It is important to note that a Lie bialgebra and its dual play symmetric roles: for any Lie bialgebra $(\g,[\cdot,\cdot],\xi)$ there exists a dual Lie bialgebra $(\g^\ast,[\cdot,\cdot]_*,\xi^*)$ where the Lie algebra structure on $\g^\ast$ is defined by $\xi^t=[\cdot,\cdot]_*$, and the dual $1$-cocycle $\xi^*$ is given by dualizing the Lie bracket in $\g$, i.e. 
\[\xi^*(\alpha)\left(x\wedge y\right)=\alpha\left([x,y]\right)\ \text{for all}\ \alpha\in\g^\ast\ \text{and all}\ x,y\in\g.\]
More details on Lie bialgebras can be found in \cite{YKS}.
\subsection{Contravariant connections and Metacurvature}
Let $(M,\pi)$ be a Poisson manifold. The associated Koszul Lie bracket on $1$-forms is defined by (\cite{CFM}, page 28)
\[[\alpha,\beta]_\pi=\mathscr{L}_{\pi_\sharp(\alpha)}\beta-\mathscr{L}_{\pi_\sharp(\beta)}\alpha-d\left(\pi(\alpha,\beta)\right),\]
where $\pi_\sharp$ is the \emph{ancor map} $\pi_\sharp : T^\ast M\to TM$, $\beta\left(\pi_\sharp(\alpha)\right):=\pi(\alpha,\beta)$.\\
A contravariant connection on a $(M,\pi)$ is an $\mathbb{R}$-bilinear map $D:\Omega^1(M)\times\Omega^1(M)\to\Omega^1(M)$, $(\alpha,\beta)\mapsto D_\alpha\beta$ such that (\cite{BB}, page 444):
\begin{enumerate} 
\item $D_\alpha\beta$ is linear over $C^\infty(M)$ with respect to $\alpha$, that is
\[D_{f\alpha}\beta=fD_\alpha\beta\ \text{for all}\ f\in C^\infty(M),\]
\item $D$ satisfies the product rule
\[D_\alpha(f\beta)=fD_\alpha\beta+\pi_\sharp(\alpha)(f)\beta\ \text{for all}\ f\in
C^\infty(M).\]
\end{enumerate}
The torsion and the curvature of a contravariant connection $D$ are formally identical to the usual definitions:
\begin{align*}
T(\alpha,\beta)=& D_\alpha\beta-D_\beta\alpha-[\alpha,\beta]_\pi\\
R(\alpha,\beta)=& D_\alpha D_\beta-D_\beta D_\alpha- D_{[\alpha,\beta]_\pi}\ \text{for all}\ \alpha,\beta\in\Omega^1(M).
\end{align*}
We will call $D$ torsion-free (resp. flat) if $T$ (resp. $R$) vanishes identically.\\
The Levi-Civita contravariant connection associated naturally to
$(\pi,\langle\,,\rangle)$ can defined by the Koszul formula
\begin{align*}
2\langle D_\alpha\beta,\gamma\rangle=&\ \pi_{\#}(\alpha)\cdot\langle\beta,\gamma\rangle+\pi_{\#}(\beta)\cdot\langle \alpha,\gamma\rangle-
\pi_{\#}(\gamma)\cdot\langle\alpha,\beta\rangle\\
&\ +\langle[\gamma,\alpha]_\pi,\beta\rangle+\langle[\gamma,\beta]_\pi,\alpha\rangle+\langle[\alpha,\beta]_\pi,\gamma\rangle.
\end{align*}
\subsection{Metacurvature}\label{meta}
Let $(M,\pi)$ be a Poisson manifold and $D$ a torsion-free and flat contravariant connection with respect to $\pi$. Then, there exists a bracket $\{\;,\;\}$ on the differential graded algebra of
differential forms $\Omega^*(M)$ such that (\cite{Ha2}, page 390): 
\begin{enumerate} 
\item
$\{\;,\;\}$ is $\mathbb{R}$-bilinear antisymmetric of degree $0$, i.e.
\[\{\sigma,\rho\}=-(-1)^{\deg\sigma\deg\rho}\{\rho,\sigma\}.\]
\item The differential $d$ is a derivation with respect to
$\{\;,\;\}$, i.e.
\[d\{\sigma,\rho\}=\{d\sigma,\rho\}+(-1)^{\deg\sigma}\{\sigma,d\rho\}.\]
\item $\{\;,\;\}$ satisfies the product rule
\[\{\sigma,\rho\wedge\lambda\}=\{\sigma,\rho\}\wedge\lambda+(-1)^{\deg\sigma
\deg\rho}\rho\wedge\{\sigma,\lambda\}.\]
\item For any $f,h\in
C^\infty(M)$ and for any $\sigma\in\Omega^*(P)$ the bracket $\{f,g\}$
coincides with the initial Poisson bracket and
\[\{f,\sigma\}=D_{df}\sigma.\]
\end{enumerate}
Hawkins called this bracket a generalized Poisson bracket and showed that there exists a $(2,3)$-tensor $M$ such that the following assertions are equivalent:
\begin{enumerate}
\item The generalized Poisson bracket satisfies the graded Jacobi identity
\[\{\{\sigma,\rho\},\lambda\}=\{\sigma,\{\rho,\lambda\}\}
-(-1)^{\deg\sigma\deg\rho}\{\rho,\{\sigma,\lambda\}\}.\]
\item The tensor $M$ vanishes identically.
\end{enumerate}
$M$ is called the metacurvature and is given by
\[M(df,\alpha,\beta)=\{f,\{\alpha,\beta\}\}-\{\{f,\alpha\},\beta\}-
\{\{f,\beta\},\alpha\}.\]
Hawkins pointed out in \cite{Ha2} page 393, that for any parallel $1$-form $\alpha$ and any $1$-form $\beta$, the generalized Poisson bracket of $\alpha$ and $\beta$ is given by
\[\{\alpha,\beta\}=-D_{\beta}d\alpha.\]
Then, one can deduce that for any parallel $1$-form $\alpha$ and for any forms $\beta,\gamma$ we have
\[M(\alpha,\beta,\gamma)=-D_\beta D_\gamma d\alpha.\]
\subsection{Flat Lie algebras}
A Lie algebra $\g$ (real of finite dimension) is said to be \emph{flat} if it is endowed with a positive definite scalar product $\langle\,,\rangle$ for which the associated infinitesimal Levi-Civita connection, defined for all $x,y,z$ in $\g$ by
\begin{equation}
2\langle\nabla_xy,z\rangle=\langle[x,y],z\rangle+\langle[z,x],y\rangle+\langle[z,y],x\rangle
\end{equation}
has zero curvature
\begin{equation}
R(x,y,z)=\nabla_{[x,y]}z-\left(\nabla_x\nabla_yz-\nabla_y\nabla_xz\right)\equiv0.\label{flat}
\end{equation}
In other words, the associated Lie group $G$ of $\g$, endowed with the unique left invariant Riemannian metric extending $\langle\,,\rangle$ has its Levi-Civita connection flat.\par 
Milnor in \cite{MIL} characterized these flat Lie algebras. Some refinements have been provided later in \cite{BB} and \cite{BDF}.
\begin{prop}[\cite{BB}, \cite{BDF}, \cite{MIL}]\label{flat}
Let $(\g,\langle\,,\rangle)$ be a flat Lie algebra. Then
$\g$ decomposes orthogonally as 
\[\g=\s\oplus\z\oplus\d,\] 
where $\z$ is the center of $\g$, $\s$ is an abelian Lie subalgebra, $\d$ is the commutator ideal satisfying the following conditions:
\begin{itemize}
\item $\d$ is abelian and even dimensional,
\item $\ad_x=\nabla_x$, for any $x$ in $\z\oplus\s$.
\end{itemize}
\end{prop}
\begin{rem}\label{k_unimodular}
It follows from the above proposition that $\g$ is a unimodular $2$-step solvable Lie algebra, whose nilradical is given by $\z\oplus\d$.
\end{rem}
Moreover, from \cite{BB}, we have:
\[\g=\spann\{s_1,\ldots,s_{k_0},z_1,\ldots,z_{\ell_0}\}\oplus\spann\{d_1,\ldots,d_{2m}\}\]
where $\spann\{d_1,\ldots,d_{2m}\}$ is the commutator ideal of $\g$ which is abelian, $\spann\{z_1,\ldots,z_{\ell_0}\}$ its center (possibly trivial) and $\spann\{s_1,\ldots,s_{k_0}\}$ its abelian subalgebra such that $k_0\leq m$ and
\begin{equation*}
[s_i,d_{2j-1}]=\lambda_{ij}\,d_{2j},\  [s_i,d_{2j}]=-\lambda_{ij}\,d_{2j-1}\ \text{for all}\ i=1,\ldots,k_0,\,\text{and}\ j=1,\ldots,m.
\end{equation*}
Indeed, the family $\{\ad_s:  s \in \s\} \subseteq \frak{so}(2m)$ is an abelian subalgebra, then it is conjugate by an element in $SO(2m)$ to a subalgebra of the  maximal abelian subalgebra of $\frak{so}(2m) $
\[  \frak t ^m= \left\{  \begin{pmatrix} 0 & -\lambda_1 & & &  \\
	\lambda_1 &0 & & & \\
	&   &   \ddots &  & \\
	&  &  & 0& -\lambda_m \\
	&  &  & \lambda_m & 0
\end{pmatrix} ,\ \lambda_i \in \R  \right\}           \]
with respect to an \textbf{orthonormal basis} $\{d_1 , \ldots , d_{2m}\}$  of $\d$. Since $\ad_{s_k}:\d\to\d$ is real skew-symmetric, its nonzero eigenvalues are all pure imaginary and thus are of the form $i\lambda_{k1},-i\lambda_{k1},\ldots,i\lambda_{km}, -i\lambda_{km}$. The endomorphism $\ad_{s_k}$ can be written in the form $\ad_{s_k}=OJ(\lambda_{k_1},\ldots,\lambda_{km})O^t$, where $O$ is an orthogonal matrix and
\[J(\lambda_{k1},\ldots,\lambda_{km})= \begin{pmatrix}
	\begin{matrix}0 & \lambda_{k1} \\ -\lambda_{k1} & 0\end{matrix} &  0 & \cdots & 0 \\
	0 & \begin{matrix}0 & \lambda_{k2} \\ -\lambda_{k2} & 0\end{matrix} & & 0 \\
	\vdots & & \ddots & \vdots \\
	0 & 0 & \cdots & \begin{matrix}0 & \lambda_{km}\\ -\lambda_{km} & 0\end{matrix}
\end{pmatrix}.\]
Two flat Lie algebras $\g_1$ and $\g_2$ are isomorphic if there exists a map $\varphi : \g_1\to\g_2$ which is simultaneously an isomorphism of Lie algebras and of Euclidean spaces (i.e. an isometry).
\begin{ejemplo}~
\begin{enumerate}
\item Any commutative Lie algebra is flat.
\item The Lie algebra $\g=\spann\{s\}\oplus\spann\{d_1,d_2\}$ with the brackets
\[[s,d_1]=d_2,\ [s,d_2]=-d_1,\ [d_1,d_2]=0,\]
is the smallest non abelian flat Lie algebra.
\end{enumerate}
\end{ejemplo}
\subsection*{Nondegenerate flat Lie algebras}
A flat Lie algebra is \emph{nondegenerate} if, \textbf{there exists a basis} $\{s_1,\ldots,s_{k_0}\}$ of $\s$ such that
\begin{equation}\label{ndeg}
\text{for all}\ i,j\in\{1,\ldots,m\},\,i\neq j,\ \text{there exists}\ k\in\{1,\ldots,k_0\},\ \text{such that}\ \lambda_{kj}^2\neq\lambda_{ki}^2.
\end{equation}
Otherwise, the Lie algebra is \emph{degenerate}.
\begin{ejemplo}~
\begin{enumerate}
\item The flat Lie algebra $\g=\spann\{s\}\oplus\spann\{d_1,d_2,d_3,d_4\}$ with the brackets 
\[[s,d_1]=d_2,\ [s,d_2]=-d_1,\ [s,d_3]=\alpha d_4,\ [s,d_4]=-\alpha d_3,\ \alpha\neq 0,-1,1\]
is nondegenerate.
\item The flat Lie algebra $\g=\spann\{s_1,s_2\}\oplus\spann\{d_1,d_2,d_3,d_4\}$ with the brackets
\[[s_i,d_{2j-1}]=\delta_{ij}\,d_{2j},\,[s_i,d_{2j}]=-\delta_{ij}\,d_{2j-1},\ i,j=1,2,\]
where $\delta$ stands for the Kronecker symbol, is nondegenerate. Although for the basis $\{s_1+s_2,s_1-s_2\}$ of $\s$ the property \eqref{ndeg} is not satisfied.
\item The flat Lie algebra $\g=\spann\{s\}\oplus\spann\{d_1,\ldots,d_4\}$ with the brackets \[[s,d_1]=d_2,\,[s,d_2]=-d_1,\,[s,d_3]=d_4,\,[s,d_4]=-d_3,\]
is degenerate, since $\lambda_{12}=\lambda_{13}$ and \eqref{ndeg} is not satisfied for any other vector generating $\s$.
\end{enumerate}
\end{ejemplo}
In the following, we will deduce the necessary (and sufficient) condition for the degeneracy property to be satisfied regardless of the choice of a basis of $\s$.\par
Let $\{s_1,\dotsc,s_{k_0}\}$ be a fixed basis of $\s$ and let $\lambda_1,\dots,\lambda_m$ be linear forms on $\s$ such that $\lambda_i(s_k)=\lambda_{ki}$ for any $1\leq i\leq m$.
Denote by $C^\infty(\s)$ the algebra of smooth functions, endowed with the pointwise product. We have for any $1\leq i,j\leq m$ the quadratic form 
\[q_{ij}=\lambda_i^2-\lambda_j^2.\]
If $(s_1^*,\dots,s_{k_0}^*)$ is the dual basis of $\s^*$ then $\lambda_i = \sum_{k=1}^{k_0}\lambda_{ki}\,s_k^*$ and
\[\lambda_i^2=\sum_{k=1}^{k_0}\lambda_{ki}^2\,s_k^{\ast2}+2\sum_{1\leq p<q\leq k_0}\lambda_{pi}\lambda_{qi}\,s_p^*s_q^*.\]
The condition $\lambda_{ki}^2=\lambda_{kj}^2$ for all $1\leq k\leq k_0$ means
\[q_{ij}=2\sum_{1\leq p<q\leq k_0}\left(\lambda_{pi}\lambda_{qi}-\lambda_{pj}\lambda_{qj}\right)s_p^\ast s_q^*.\]
If $g\in\mathrm{End}(\s)$ then $gs_k=\sum_{i=1}^{k_0} a_{ik}s_i$ for some $a_{ik}\in\mathbb{R}$ and $s_i^*(gs_k)s_j^*(gs_k)=a_{ik}a_{jk}$ so 
\[q_{ij}(gs_k)=2\sum_{1\leq p<q\leq k_0}a_{pk}a_{qk}\left(\lambda_{pi}\lambda_{qi}-\lambda_{pj}\lambda_{qj}\right).\]
We will show that $q_{ij}=0$ given that $q_{ij}(gs_k)=0$ for all $g \in\mathrm{GL}(\s)$ and $1\leq k\leq k_0$. For any $1\leq p,q\leq k_0$ the endomorphism $\tau_{pq}$ of $V$ given by $\tau_{pq}(s_k)=s_k+\delta_{kp}s_q$ is invertible (where $\delta$ stands for the Kronecker symbol) and $\tau_{pq}(s_p)=s_p+s_q$. Hence $q_{ij}(s_p+s_q)=q_{ij}(\tau_{pq}s_p)=0$ for any $1\leq p,q\leq k_0$ (if $p=q$ then $q_{ij}(2s_p)=4q_{ij}(s_p) =0$).

Now let $\beta(u,v)=q_{ij}(u+v)-q_{ij}(u)-q_{ij}(v)$ be the symmetric bilinear form determined by $q_{ij}$. Then for any $1\leq p,q\leq k_0$ we have
\[\beta(s_p,s_q)=q_{ij}(s_p+s_q)-q_{ij}(s_p)-q_{ij}(s_q)=0.\]
By bilinearity we deduce that $\beta\equiv 0$ which shows that $q_{ij}(v)=\frac{1}{2}\beta(v,v)=0$ for any $v\in V$. Thus $q_{ij}$ is identically zero, i.e.
\[\lambda_{pi}\lambda_{qi}-\lambda_{pj}\lambda_{qj}=0\ \text{for any}\ 1\leq p<q\leq k_0.\]
Since $\lambda_{pi}=\varepsilon_{ij}^p\lambda_{pj}$ and $\lambda_{qi}=\varepsilon_{ij}^q\lambda_{qj}$ ($\varepsilon_{ij}^q=\pm1$, $\varepsilon_{ij}^q=\pm1$), then
\begin{align*}
\lambda_{pi}\lambda_{qi}-\lambda_{pj}\lambda_{qj}=0&\Longleftrightarrow\lambda_{pi}\lambda_{qi}\left(1-\varepsilon_{ij}^p\varepsilon_{ij}^q\right)=0.\\
&\Longleftrightarrow\varepsilon_{ij}^p=\varepsilon_{ij}^q=\pm1\ \text{or}\ \lambda_{pi}=\lambda_{pj}=0,\ \text{or}\ \lambda_{qi}=\lambda_{qj}=0.
\end{align*}
We deduce the following characterization. 
\begin{prop}\label{degen}
A flat Lie algebra $\g=\s\oplus\z\oplus\d$ is degenerate if and only if for any basis of $\s$ there is $1\leq i<j\leq m$ such that for all $k=1,\dotsc,k_0$, $\lambda_{ki}=\varepsilon_{ij}\lambda_{kj}$, with the property
\begin{equation}\label{deg}
\varepsilon_{ij}^k=\varepsilon_{ij}=\pm1\ \text{independently of}\ k.
\end{equation}
\end{prop}
\subsection{Metaflat Lie algebras}
Let $(\g,\langle\,,\rangle,\xi)$ be a flat Lie algebra endowed with a $1$-cocycle $\xi : \g\to\wedge^2\g$.
The flat Lie algebra $(\g,\langle\,,\rangle,\xi)$ is \emph{metaflat} if
\begin{equation}\label{meta}
\ad_x\ad_y\xi(z)=0\ \text{for all}\ x,y,z\in\s.
\end{equation}
This is equivalent to the nullity of the metacurvature \ref{meta}. This has been proved in \cite{BB}.\\
A commutative Lie algebra is obviously flat and metaflat. In the sequel, we assume that the Lie algebra is not commutative.
\section{Main theorem}
In this section we prove the main theorem (i.e., Theorem \ref{main}).\\
Let $r\in\wedge^2\g$ be a bivector and let $\xi$ be the associated $1$-coboundary. We have
\begin{align*}
r=&\sum_{1\leq i<j\leq k_0}a_{ij}\,s_i\wedge s_j+\sum_{\substack{1\leq i\leq k_0\\ 1\leq j\leq\ell_0}}b_{ij}\,s_i\wedge z_j+\sum_{\substack{1\leq i\leq k_0\\ 1\leq j\leq m}}c_{ij}\,s_i\wedge d_{2j-1}+\sum_{\substack{1\leq i\leq k_0\\ 1\leq j\leq m}}e_{ij}\,s_i\wedge d_{2j}\\
&+\sum_{1\leq i<j\leq\ell_0}f_{ij}\,z_i\wedge z_j+\sum_{\substack{1\leq i\leq\ell_0\\ 1\leq j\leq m}}g_{ij}\,z_i\wedge d_{2j-1}+\sum_{\substack{1\leq i\leq\ell_0\\ 1\leq j\leq m}}h_{ij}\,z_i\wedge d_{2j}\\
&+\sum_{1\leq i<j\leq m}m_{ij}\,d_{2i-1}\wedge d_{2j-1}+\sum_{1\leq i,j\leq m}n_{ij}\,d_{2i-1}\wedge d_{2j}+\sum_{1\leq i<j\leq m}p_{ij}\,d_{2i}\wedge d_{2j}.
\end{align*}
For all $k=1,\ldots,k_0$, $\xi(s_k)=\ad_{s_k}r$, thus
\begin{align}
\xi(s_k)=&\sum_{\substack{1\leq i\leq k_0\\ 1\leq j\leq m}}-\lambda_{kj}e_{ij}\,s_i\wedge d_{2j-1}+\sum_{\substack{1\leq i\leq k_0\\ 1\leq j\leq m}}\lambda_{kj}c_{ij}\,s_i\wedge d_{2j}+\sum_{\substack{1\leq i\leq\ell_0\\ 1\leq j\leq m}}-\lambda_{kj}h_{ij}\,z_i\wedge d_{2j-1}\notag\\
&+\sum_{\substack{1\leq i\leq\ell_0\\ 1\leq j\leq m}}\lambda_{kj}g_{ij}\,z_i\wedge d_{2j}+\sum_{1\leq i<j\leq m}\left(-\lambda_{kj}n_{ij}+\lambda_{ki}n_{ji}\right)\,d_{2i-1}\wedge d_{2j-1}\notag\\
&+\sum_{1\leq i<j\leq m}\left(\lambda_{kj}m_{ij}-\lambda_{ki}p_{ij}\right)\,d_{2i-1}\wedge d_{2j}+\sum_{1\leq i<j\leq m}\left(-\lambda_{ki}m_{ij}+\lambda_{kj}p_{ij}\right)\,d_{2j-1}\wedge d_{2i}\notag\\
&+\sum_{1\leq i<j\leq m}\left(\lambda_{ki}n_{ij}-\lambda_{kj}n_{ji}\right)\,d_{2i}\wedge d_{2j}.\label{cocycle}
\end{align}
For all $k=1,\ldots,\ell_0$, $\xi(z_k)=\ad_{z_k}r=0$ and for all $k=1,\ldots,m$
\[\xi(d_{2k-1})=\ad_{d_{2k-1}}r=\Phi_k\wedge d_{2k},\ \xi(d_{2k})=\ad_{d_{2k}}r=d_{2k-1}\wedge\Phi_k\]
where 
\begin{align*}
\Phi_k=&\left(\sum_{j=2}^{k_0}-\lambda_{jk}a_{1j}\right)s_1+\sum_{p=2}^{k_0-1}\left(\sum_{i=1}^{p-1}\lambda_{ik}a_{ip}+\sum_{j=p+1}^{k_0}-\lambda_{jk}a_{pj}\right)s_p+\left(\sum_{i=1}^{k_0-1}\lambda_{ik}a_{ik_0}\right)s_{k_0}\\
&+\sum_{j=1}^{\ell_0}\left(\sum_{i=1}^{k_0}\lambda_{ik}b_{ij}\right)z_j+\sum_{j=1}^m\left(\sum_{i=1}^{k_0}\lambda_{ik}c_{ij}\right)d_{2j-1}+\sum_{j=1}^m\left(\sum_{i=1}^{k_0}\lambda_{ik}e_{ij}\right)d_{2j}.
\end{align*}
\begin{prop}
The \emph{metaflatness} condition \ref{meta} gives
\[\xi(s_k)=0\ \text{for all}\ k=1,\ldots,k_0.\]
\end{prop}
\begin{proof}
Since $\xi$ is a coboundary, we can simply write
\begin{align*}
\xi(s_k)=&\sum_{\substack{1\leq i\leq k_0\\ 1\leq j\leq m}}c_{ij}^k\,s_i\wedge d_{2j-1}+\sum_{\substack{1\leq i\leq k_0\\ 1\leq j\leq m}}e_{ij}^k\,s_i\wedge d_{2j}+\sum_{\substack{1\leq i\leq\ell_0\\ 1\leq j\leq m}}g_{ij}^k\,z_i\wedge d_{2j-1}\\
&+\sum_{\substack{1\leq i\leq\ell_0\\ 1\leq j\leq m}}h_{ij}^k\,z_i\wedge d_{2j}+\sum_{1\leq i<j\leq m}m_{ij}^k\,d_{2i-1}\wedge d_{2j-1}\\
&+\sum_{1\leq i,j\leq m}n_{ij}^k\,d_{2i-1}\wedge d_{2j}+\sum_{1\leq i<j\leq m}p_{ij}^k\,d_{2i}\wedge d_{2j},
\end{align*}
and
\begin{align*}
\ad_{s_\ell}^2\xi(s_k)=&\sum_{\substack{1\leq i\leq k_0\\ 1\leq j\leq m}}-\lambda_{\ell j}^2 c_{ij}^k\,s_i\wedge d_{2j-1}+\sum_{\substack{1\leq i\leq k_0\\ 1\leq j\leq m}}-\lambda_{\ell j}^2 e_{ij}^k\,s_i\wedge d_{2j}\\
&+\sum_{\substack{1\leq i\leq\ell_0\\ 1\leq j\leq m}}-\lambda_{\ell j}^2 g_{ij}^k\,z_i\wedge d_{2j-1}+\sum_{\substack{1\leq i\leq\ell_0\\ 1\leq j\leq m}}-\lambda_{\ell j}^2 h_{ij}^k\,z_i\wedge d_{2j}\\
&+\sum_{1\leq i<j\leq m}\left(-(\lambda_{\ell i}^2+\lambda_{\ell j}^2)m_{ij}^k+2\lambda_{\ell i}\lambda_{\ell j}p_{ij}^k\right)\,d_{2i-1}\wedge d_{2j-1}\\
&+\sum_{1\leq i<j\leq m}\left(-(\lambda_{\ell i}^2+\lambda_{\ell j}^2)n_{ij}^k+2\lambda_{\ell i}\lambda_{\ell j}n_{ji}^k\right)\,d_{2i-1}\wedge d_{2j}\\
&+\sum_{1\leq i<j\leq m}\left(2\lambda_{\ell i}\lambda_{\ell j}n_{ij}^k-(\lambda_{\ell i}^2+\lambda_{\ell j}^2)n_{ji}^k\right)\,d_{2j-1}\wedge d_{2i}\\
&+\sum_{1\leq i<j\leq m}\left(2\lambda_{\ell i}\lambda_{\ell j}m_{ij}^k-(\lambda_{\ell i}^2+\lambda_{\ell j}^2)p_{ij}^k\right)\,d_{2i}\wedge d_{2j}.
\end{align*}
So that $\ad_{s_\ell}^2\xi(s_k)=0$ for all $k,\ell\in\{1,\ldots,k_0\}$ if and only if 
\[c_{ij}^k=e_{ij}^k=g_{ij}^k=h_{ij}^k=0\]
and
\begin{equation*}
\begin{pmatrix}
-(\lambda_{\ell i}^2+\lambda_{\ell j}^2)&\phantom{-}2\lambda_{\ell i}\lambda_{\ell j}\\
\phantom{-}2\lambda_{\ell i}\lambda_{\ell j}&-(\lambda_{\ell i}^2+\lambda_{\ell j}^2)
\end{pmatrix}\begin{pmatrix}
m_{ij}^k\\
p_{ij}^k
\end{pmatrix}=
\begin{pmatrix}
0\\
0
\end{pmatrix}=\begin{pmatrix}
-(\lambda_{\ell i}^2+\lambda_{\ell j}^2)&\phantom{-}2\lambda_{\ell i}\lambda_{\ell j}\\
\phantom{-}2\lambda_{\ell i}\lambda_{\ell j}&-(\lambda_{\ell i}^2+\lambda_{\ell j}^2)
\end{pmatrix}\begin{pmatrix}
n_{ij}^k\\
n_{ji}^k
\end{pmatrix}.
\end{equation*}
If $\g$ is non degenerate, i.e. for all $1\leq i<j\leq m$ there is $\ell\in\{1,\ldots,k_0\}$ such that  the determinant $(\lambda_{\ell i}^2-\lambda_{\ell j}^2)^2\neq0$, then
\[m_{ij}^k=p_{ij}^k=n_{ij}^k=n_{ji}^k=0,\ \text{and from}\ \eqref{cocycle}, \ \xi(s_k)=0.\]
Suppose $\g$ is degenerate, i.e. for all $\ell\in\{1,\ldots,k_0\}$ $\lambda_{\ell j}=\varepsilon_{ij}\lambda_{\ell i}$, ($\varepsilon_{ij}=\pm1$). We can choose $\ell$ such that $\lambda_{\ell i}\neq0$. Otherwise, $d_{2i}$ would be in the center, which is impossible. Then from the characterization \eqref{deg} of $\g$ we get, for all $1\leq i<j\leq m$, $p_{ij}^k=\varepsilon_{ij}m_{ij}^k$ and $n_{ji}^k=\varepsilon_{ij}n_{ij}^k$. By \eqref{cocycle}, we get the lemma, since for a degenerate flat algebra, we have
\[\lambda_{kj}n_{ij}-\lambda_{ki}n_{ji}=\left(\varepsilon_{ij}\lambda_{ki}\right)n_{ij}-\lambda_{ki}\left(\varepsilon_{ij}n_{ij}\right)=0\]
and similarly for the other coefficients.
\end{proof}
As a direct consequence, we get the following proposition:
\begin{prop}
Let $\g$ be a flat and metaflat Lie algebra and $r\in\wedge^2\g$.
\begin{enumerate}
\item If $\g$ is nondegenerate, then 
\begin{align*}
r=&\sum_{1\leq i<j\leq k_0}a_{ij}\,s_i\wedge s_j+\sum_{\substack{1\leq i\leq k_0\\ 1\leq j\leq\ell_0}}b_{ij}\,s_i\wedge z_j+\sum_{1\leq i<j\leq\ell_0}f_{ij}\,z_i\wedge z_j+\sum_{1\leq i\leq m}n_{ii}\,d_{2i-1}\wedge d_{2i}
\end{align*}
\item If $\g$ is degenerate then, 
\begin{align*}
r=&\sum_{1\leq i<j\leq k_0}a_{ij}\,s_i\wedge s_j+\sum_{\substack{1\leq i\leq k_0\\ 1\leq j\leq\ell_0}}b_{ij}\,s_i\wedge z_j+\sum_{1\leq i<j\leq\ell_0}f_{ij}\,z_i\wedge z_j\\
&+\sum_{1\leq i<j\leq m}m_{ij}\,d_{2i-1}\wedge d_{2j-1}+\sum_{1\leq i,j\leq m}n_{ij}\,d_{2i-1}\wedge d_{2j}+\sum_{1\leq i<j\leq m}p_{ij}\,d_{2i}\wedge d_{2j}
\end{align*}
such that $p_{ij}=\varepsilon_{ij}m_{ij}$ and $n_{ji}=\varepsilon_{ij} n_{ji}$ ($\varepsilon_{ij}=\pm1$).
\end{enumerate}
\end{prop}
We can now state and prove the main theorem.
\begin{teo}\label{main}
Let $\g$ be a flat Lie algebra and let $\xi$ be a $1$-coboundary associated to some bivector $r\in\wedge^2\g$. If $(\g,\xi)$ is metaflat, then $r$ is a solution of the classical Yang-Baxter equation $[r,r]=0$. Moreover, the dual Lie bialgebra $\g^\ast$ is flat and metaflat for the dual positive-definite scalar product.
\end{teo}
\begin{proof}
In both the degenerate and nondegenerate cases, we can assume 
\begin{align*}
r=&\sum_{1\leq i<j\leq k_0}a_{ij}\,s_i\wedge s_j+\sum_{\substack{1\leq i\leq k_0\\ 1\leq j\leq\ell_0}}b_{ij}\,s_i\wedge z_j+\sum_{1\leq i<j\leq\ell_0}f_{ij}\,z_i\wedge z_j\\
&+\sum_{1\leq i<j\leq m}m_{ij}\,d_{2i-1}\wedge d_{2j-1}+\sum_{1\leq i,j\leq m}n_{ij}\,d_{2i-1}\wedge d_{2j}+\sum_{1\leq i<j\leq m}p_{ij}\,d_{2i}\wedge d_{2j}
\end{align*}
Recall that the Schouten-Nijenhuis bracket is $\mathbb{R}$-bilinear with respect to its two arguments and is defined decomposable mutlivectors as follows \cite{CFM}
\begin{multline*}
[X_1\wedge\cdots\wedge X_m,Y_1\wedge\cdots\wedge Y_n]=\\
\sum_{i,j}(-1)^{i+j}[X_i,Y_j]\wedge X_1\wedge\cdots\wedge \widehat{X_i}\wedge\cdots\wedge X_m\wedge Y_1\wedge\cdots\wedge\widehat{Y}_j\wedge\cdots\wedge Y_n
\end{multline*}
where the hat sign indicates that the argument below has been omitted.\par 
For all vector fields $X_1,X_2,Y_1,Y_2$ we have
\[[X_1\wedge X_2,Y_1\wedge Y_2]=[X_1,Y_1]\wedge X_2\wedge Y_2-[X_1,Y_2]\wedge X_2\wedge Y_1-[X_2,Y_1]\wedge X_1\wedge Y_2+[X_2,Y_2]\wedge X_1\wedge Y_1.\]
Moreover, the bracket is symmetric when restricted to bivectors
\[[X_1\wedge X_2,Y_1\wedge Y_2]=[Y_1\wedge Y_2,X_1\wedge X_2].\]
Therefore we have
\begin{align*}
\frac{1}{2}[r,r]=&\sum_{1\leq i<j\leq k_0}\sum_{1\leq k<\ell\leq m}a_{ij}m_{kl}\,[s_i\wedge s_j,d_{2k-1}\wedge d_{2\ell-1}]+\sum_{1\leq i<j\leq k_0}\sum_{1\leq k,\ell\leq m}a_{ij}n_{kl}\,[s_i\wedge s_j,d_{2k-1}\wedge d_{2\ell}]\\
&+\sum_{1\leq i<j\leq k_0}\sum_{1\leq k<\ell\leq m}a_{ij}p_{kl}\,[s_i\wedge s_j,d_{2k}\wedge d_{2\ell}]+\sum_{\substack{1\leq i\leq k_0\\ 1\leq j\leq\ell_0}}\sum_{1\leq k<\ell\leq m}b_{ij}m_{kl}\,[s_i\wedge z_j,d_{2k-1}\wedge d_{2\ell-1}]\\
&+\sum_{\substack{1\leq i\leq k_0\\ 1\leq j\leq\ell_0}}\sum_{1\leq k,\ell\leq m}b_{ij}n_{kl}\,[s_i\wedge z_j,d_{2k-1}\wedge d_{2\ell}]+\sum_{\substack{1\leq i\leq k_0\\ 1\leq j\leq\ell_0}}\sum_{1\leq k<\ell\leq m}b_{ij}m_{kl}\,[s_i\wedge z_j,d_{2k}\wedge d_{2\ell}].
\end{align*}
By using $\Gamma_{ij}^\ell=\lambda_{j\ell}s_i-\lambda_{i\ell}s_j$, we have
\begin{align*}
\frac{1}{2}[r,r]=&\sum_{1\leq i<j\leq k_0}\sum_{1\leq k<\ell\leq m}a_{ij}m_{kl}\left(\Gamma_{ij}^\ell\wedge d_{2k-1}\wedge d_{2\ell}-\Gamma_{ij}^k\wedge d_{2\ell-1}\wedge d_{2k}\right)\\
&+\sum_{1\leq i<j\leq k_0}\sum_{1\leq k,\ell\leq m}a_{ij}n_{kl}\left(-\Gamma_{ij}^\ell\wedge d_{2k-1}\wedge d_{2\ell-1}+\Gamma_{ij}^k\wedge d_{2k}\wedge d_{2\ell}\right)\\
&+\sum_{1\leq i<j\leq k_0}\sum_{1\leq k<\ell\leq m}a_{ij}p_{kl}\left(-\Gamma_{ij}^k\wedge d_{2k-1}\wedge d_{2\ell}+\Gamma_{ij}^\ell\wedge d_{2\ell-1}\wedge d_{2k}\right)\\
&+\sum_{\substack{1\leq i\leq k_0\\ 1\leq j\leq\ell_0}}\sum_{1\leq k<\ell\leq m}b_{ij}m_{kl}\,z_j\wedge\left(-\lambda_{i\ell}\,d_{2k-1}\wedge d_{2\ell}+\lambda_{ik}\,d_{2\ell-1}\wedge d_{2k}\right)\\
&+\sum_{\substack{1\leq i\leq k_0\\ 1\leq j\leq\ell_0}}\sum_{1\leq k,\ell\leq m}b_{ij}n_{kl}\,z_j\wedge\left(\lambda_{i\ell}\,d_{2k-1}\wedge d_{2\ell-1}-\lambda_{ik}\,d_{2k}\wedge d_{2\ell}\right)\\
&+\sum_{\substack{1\leq i\leq k_0\\ 1\leq j\leq\ell_0}}\sum_{1\leq k<\ell\leq m}b_{ij}p_{kl}\,z_j\wedge\left(\lambda_{ik}\,d_{2k-1}\wedge d_{2\ell}-\lambda_{i\ell}\,d_{2\ell-1}\wedge d_{2k}\right).
\end{align*}
Now, we show that the 2nd and the 5th terms are zero and by combining the terms 1st,4th with 3rd,6th respectively, we get $[r,r]=0$. We also show that above equality always holds, whether $\g$ is degenerate or not.
If $\g$ is nondegenerate, then
\begin{align*}
[r,r]=&\sum_{1\leq i<j\leq k_0}\sum_{1\leq k\leq m}a_{ij}n_{kk}\left(-\Gamma_{ij}^k\wedge d_{2k-1}\wedge d_{2k-1}+\Gamma_{ij}^k\wedge d_{2k}\wedge d_{2k}\right)\\
&+\sum_{\substack{1\leq i\leq k_0\\ 1\leq j\leq\ell_0}}\sum_{1\leq k\leq m}b_{ij}n_{kk}\,z_j\wedge\left(\lambda_{ik}\,d_{2k-1}\wedge d_{2k-1}-\lambda_{ik}\,d_{2k}\wedge d_{2k}\right)=0.
\end{align*}
If $\g$ is degenerate, then for all $k<\ell$ in $\{1,\ldots,m\}$ and for all $i,j$ in $\{1,\ldots,k_0\}$,
\begin{equation*}
\lambda_{i\ell}=\varepsilon_{k\ell}\,\lambda_{ik},\ \lambda_{j\ell}=\varepsilon_{k\ell}\,\lambda_{jk},\ p_{k\ell}=\varepsilon_{k\ell}\,m_{k\ell},\ n_{\ell k}=\varepsilon_{k\ell}\,n_{k\ell}.
\end{equation*}
Therefore, $\Gamma_{ij}^\ell=(\varepsilon_{k\ell}\,\lambda_{jk})\,s_i-(\varepsilon_{k\ell}\,\lambda_{ik})\,s_j=\varepsilon_{k\ell}\Gamma_{ij}^k$, so
\begin{equation*}
\begin{array}{rccc}
\phantom{-}m_{k\ell}\Gamma_{ij}^\ell-p_{k\ell}\Gamma_{ij}^k&=\phantom{-}m_{k\ell}(\varepsilon_{k\ell}\,\Gamma_{ij}^k)-(\varepsilon_{k\ell}\,m_{k\ell})\Gamma_{ij}^k&=&\ 0,\\
-m_{k\ell}\Gamma_{ij}^k+p_{k\ell}\Gamma_{ij}^\ell&=-m_{k\ell}\Gamma_{ij}^k+(\varepsilon_{k\ell}\,m_{k\ell})(\varepsilon_{k\ell}\Gamma_{ij}^k)&=&\ 0.
\end{array}
\end{equation*}
Thus the sum of the 1st term with the 3rd term is zero. For the same reasons, combining the other terms gives zero.\par 
Let $\g^\ast$ be the dual vector space with dual basis
\[\left\{s_1^\ast,\ldots,s_{k_0}^\ast,z_1^\ast,\ldots,z_{\ell_0},d_1^\ast,\ldots,d_{2m}^\ast\right\}.\]
The subspace $\spann\{s_1^\ast,\ldots,s_{k_0}^\ast,z_1^\ast,\ldots,z_{\ell_0}^\ast\}$ is a commutative subalgebra, $\spann\{d_1^\ast,\ldots,d_{2m}^\ast\}$ is the commutator ideal; it is abelian and satisfies the following equations
\begin{equation*}
\begin{array}{ll}
[s_i^\ast,d_{2j-1}^\ast]=-s_i^\ast\left(\Phi_j\right)\,d_{2j}^\ast,& [s_i^\ast,d_{2j}^\ast]=s_i^\ast\left(\Phi_j\right)\,d_{2j-1}^\ast,\\\relax
[z_i^\ast,d_{2j-1}^\ast]=-z_i^\ast\left(\Phi_j\right)\,d_{2j}^\ast,& [z_i^\ast,d_{2j}^\ast]=z_i^\ast\left(\Phi_j\right)\,d_{2j-1}^\ast,
\end{array}
\end{equation*}
where 
\begin{multline*}
\Phi_k=\left(\sum_{j=2}^{k_0}-\lambda_{jk}a_{1j}\right)s_1+\sum_{p=2}^{k_0-1}\left(\sum_{i=1}^{p-1}\lambda_{ik}a_{ip}+\sum_{j=p+1}^{k_0}-\lambda_{jk}a_{pj}\right)s_p\\
+\left(\sum_{i=1}^{k_0-1}\lambda_{ik}a_{ik_0}\right)s_{k_0}+\sum_{j=1}^{\ell_0}\left(\sum_{i=1}^{k_0}\lambda_{ik}b_{ij}\right)z_j.
\end{multline*}
We denote by $\langle\,,\rangle^\ast$ the positive-definite scalar product on $\g^\ast$ associated to $\langle\,,\rangle$, via the isomorphism
\begin{equation*}
\begin{array}{cccc}
\sharp :& \g&\rightarrow&\g^\ast\\
&x&\mapsto&\sharp(x)(y):=\langle x,y\rangle,
\end{array}
\end{equation*}
by setting, $\langle\sharp(x),\sharp(y)\rangle^\ast:=\langle x,y\rangle.$
We conclude that $\g^\ast$ decomposes orthogonally
\[\g^\ast=\spann\{s_1^\ast,\ldots,s_{k_0}^\ast\}\oplus\{z_1^\ast,\ldots,z_{\ell_0}^\ast\}\oplus\spann\{d_1^\ast,\ldots,d_{2m}^\ast\},\]
i.e. $\g^\ast$ is flat. Since $\xi^\ast(x)=0$ for all $x$ in $\spann\{s_1^\ast,\ldots,s_{k_0}^\ast,z_1^\ast,\ldots,z_{\ell_0}^\ast\}$ and for all $k=1,\ldots,m$
\[\xi^\ast(d_{2k-1}^\ast)=-\Psi_k\wedge d_{2k}^\ast,\ \xi^\ast(d_{2k}^\ast)=\Psi_k\wedge d_{2k-1}^\ast\ \text{where}\ \Psi_k=\sum_{i=1}^{k_0}\lambda_{ik}s_i^\ast,\]
then the dual Lie algebra $\g^\ast$ is metaflat.
\end{proof}
\begin{rem}
The dual Lie bialgebra is not necessarily exact. For example, if $\g$ is a noncommutative flat Lie algebra with a trivial cocycle $\xi\equiv0$ then the dual Lie algebra is commutative and the dual cocyle $\xi^\ast$ is not trivial (thus not exact).
\end{rem}
\section{Examples}
\begin{enumerate}
\item For the $3$-dimensional flat Lie algebra $\g=\spann\{s\}\oplus\spann\{d_1,d_2\}$, with the brackets 
\[[s,d_1]=d_2,\ [s,d_2]=-d_1,\ [d_1,d_2]=0,\]
any bivector $r=a\,s\wedge d_1+b\,s\wedge d_2+c\,d_1\wedge d_2$ is a solution of the Yang-Baxter equation since $[r,r]=2(a^2+b^2)\,s\wedge d_1\wedge d_2$ is $\ad$-invariant.
The associated $1$-coboundary is given by
\[\xi(s)=-b\,s\wedge d_1+a\,s\wedge d_2,\ \xi(d_1)=a\,d_1\wedge d_2,\ \xi(d_2)=b\,d_1\wedge d_2.\]
We have
\[\xi\ \text{metaflat}\Longleftrightarrow a=b=0\Longleftrightarrow[r,r]=0.\]
\item \textbf{Triangular flat Lie bialgebra which is not metaflat.} The $4$-dimensional flat Lie algebra $\g=\spann\{s\}\oplus\spann\{z\}\oplus \spann\{d_1,d_2\}$, with brackets $[s,d_1]=d_2$, $[s,d_2]=-d_1$, the bivector $r=z\wedge d_1$ is a solution of the classical Yang-Baxter equation but $\ad_s^2\xi(s)=-z\wedge d_2\neq0$.
\item For any flat Lie algebra $\g=\spann\{s_1,\dotsc,s_{k_0}\}\oplus\z\oplus\d$, the bivector 
\[r=\sum_{1\leq i<j\leq k_0}a_{ij}\,s_i\wedge s_j\]
 is a solution of the classical Yang-Baxter equation (since $\s$ is a commutative subalgebra) and the associated $1$-cocycle $\xi(x)=\ad_xr$ is metaflat. The dual Lie algebra $\g^\ast$ is metaflat.
\item For any flat Lie algebra $\g=\s\oplus\z\oplus\spann\{d_1,\dotsc,d_{2m}\}$, the bivector 
\[r=\sum_{i=1}^m\lambda_i\,d_{2i-1}\wedge d_{2i}\]
is a solution of the classical Yang-Baxter equation (since $\d$ is an abelian ideal) and the associated $1$-cocycle $\xi(x)=\ad_xr$ is metaflat and the dual Lie algebra $\g^\ast$ is abelian.
\end{enumerate}
\section*{Acknowledgement} The author wishes to thank the editor and the referee for their valuable comments which improved the presentation of the paper.

\end{document}